\documentclass{birkjour}
 \newtheorem{thm}{Theorem}[section]
 \newtheorem{cor}[thm]{Corollary}
 
 \newtheorem{prop}[thm]{Proposition}
 \theoremstyle{definition}
 \newtheorem{defn}[thm]{Definition}
 \theoremstyle{remark}
 \newtheorem{rem}[thm]{Remark}
 \newtheorem*{ex}{Example}
 \numberwithin{equation}{section}

\usepackage{amsfonts}
\usepackage{amsmath}
\usepackage{amssymb}
\usepackage{amsthm}
\usepackage{graphicx}
\usepackage{tikz}
\usepackage{version}
\usepackage{bm}
\usepackage{enumerate}
\usepackage{epsfig}
\usepackage{caption}%
\usepackage{setspace}
\newcommand{\R}{{\mathbb R}}

\newcommand{\N}{\mathbb{N}}

\newcommand{\cO}{\mathcal{O}}
\newcommand{\bH}{\mathbb{H}}
\DeclareMathOperator{\Lip}{Lip}
\DeclareMathOperator{\diam}{diam}

\DeclareMathOperator{\rows}{rows}

\begin{document}

%
%
%
%
%
%
%
%
%

\title[Trees of Maps and Sequences of Maps between Spaces]
 {Attractors of Trees of Maps and of Sequences of Maps between Spaces with Applications to Subdivision}

\author[Dyn]{Nira Dyn}
\address{%
School of Mathematical Sciences\\ Tel Aviv University\\ Israel}
\email{niradyn@tau.ac.il}

\author[Levin]{David Levin}
\address{%
School of Mathematical Sciences\\ Tel Aviv University\\ Israel}
\email{levindd@gmail.com}

\author[Massopust]{Peter Massopust}
\address{%
Centre of Mathematics\\ Technical University of Munich\\ Germany}
\email{massopust@ma.tum.de}



\date{January 1, 2004}

\begin{abstract}
An extension of the Banach fixed point theorem for a sequence of maps on a complete metric space $(X,d)$ has been presented in a previous paper. It has been shown that 
backward trajectories of maps $X\to X$ converge under mild conditions and that they can generate new types of attractors such as scale dependent fractals.
Here we present two generalisations of this result and some potential applications. First, we study the structure of an infinite tree of maps $X\to X$ and discuss convergence to a unique ``attractor'' of the tree.
We also consider ``staircase'' sequences of maps, that is, we consider a countable sequence of metric spaces $\{(X_i,d_i)\}$ and an associated countable sequence of maps $\{T_i\}$, $T_i:X_{i}\to X_{i-1}$. We examine conditions for the convergence of backward trajectories of the $\{T_i\}$ to a unique attractor. An example of such trees of maps are trees of function systems leading to the construction of fractals which are both scale dependent and location dependent. The staircase structure facilitates linking all types of linear subdivision schemes to attractors of function systems.
\end{abstract}

\maketitle

\section{Introduction and Preliminaries}\label{sect1}

In a recent paper \cite{LDV}, the authors investigated the relation between non-stationary subdivision and sequences of function systems (SFSs). This study introduced the notion of backward trajectories on a metric space and a related version of the Banach fixed point theorem for sequences of maps. With this theory, limits of non-stationary subdivision schemes with masks of fixed size are related to attractors of SFSs.

The present paper is motivated by an attempt to relate more general types of subdivision processes to SFSs. In order to explain these goals we start with a short introduction to the field of subdivision schemes.
\subsection{Some Notions of Subdivision Schemes}\label{NOS}\hfill

\medskip
Subdivision schemes play an important role in Computer-Aided Geometric Design (CAGD) and in wavelets theory \cite{CDM,DL}. Here, we only consider univariate scalar binary subdivision schemes.

Given  a set of control points $p^{0}=\{p_j^0\in \mathbb{R}^m : j \in \mathbb{Z}\}$ at level $0$, a stationary binary subdivision scheme recursively defines new sets of points $p^k= \{p_j^k: j \in \mathbb{Z}\}$ at level $k \ge 1$ by the refinement rule
\begin{equation}\label{sa}
p_i^{k+1} = \sum_{j \in \mathbb{Z}} a_{i-2j} p_j^k,\quad k\ge 0,
\end{equation}
or in short form,
\begin{equation}\label{short1}
\bar{p}^{k+1}=S_a \bar{p}^k,\quad k\ge 0,
\end{equation}
The refinement is manifested by relating each point $p_j^k$ at level $k$ of the subdivision process to the binary parametric value $j2^{-k}$.
The set of real coefficients $a= \{a_j: j \in \mathbb{Z}\}$ that determines the refinement rule is called the mask of the scheme.
We assume that the support of the mask, $\sigma(a) :=
\{j \in \mathbb{Z}:a_j \neq 0\}$, is finite. In (\ref{short1}),
$S_a$ is a bi-infinite two-slanted matrix with entries $(S_a)_{i,j}=a_{i-2j}$, and $\bar{p}^k$ is an infinite matrix whose rows are the points at level $k$, $(\bar{p}^k)_{i,\cdot}=p^k_i\in\mathbb{R}^m$.

A non-stationary binary subdivision scheme is defined formally as
\begin{equation}\label{short2}
\bar{p}^{k+1}=S_{a^{[k]}} \bar{p}^k,\quad k\ge 0,
\end{equation}
where the refinement rule at refinement level $k$ is of the form
\begin{equation}\label{sak}
p_i^{k+1} = \sum_{j \in \mathbb{Z}^s} a_{i-2j}^{[k]} p_j^k,\quad i\in \mathbb{Z},
\end{equation}
i.e., in a non-stationary scheme, the mask $a^{[k]}:= \{a_j^{[k]}: j \in \mathbb{Z}\}$  depends on the refinement level $k$.

The classical definition of a convergent subdivision scheme is the following.
\begin{defn}[\textbf{$C^\nu$-Convergent Subdivision}]\label{C0conv}
A subdivision scheme is termed $C^\nu$-convergent, $\nu\ge 0$, if for any initial data  $p^0$ there exists a $C^\nu$ function $f:\mathbb{R} \to \mathbb{R}^m$ such that
\begin{equation}\label{pktof}
\lim_{k\to\infty}\sup_{i\in \mathbb{Z}}|p_i^{k}-f(2^{-k}i)|=0,
\end{equation}
and for some initial data $f\ne 0$.
\end{defn}

\begin{rem}\label{remarkC0}
\hfill

The limit curve of a $C^0$-convergent subdivision applied to initial data  $p^0$ is denoted by $p^\infty=S_a^\infty p^0$. The function $f$ in Definition \ref{C0conv} specifies a parametrization of the limit curve.
\end{rem}
The analysis of subdivision schemes aims at studying the smoothness properties of the limit function $f$. For further reading, see \cite{DL}.

A weaker type of convergence is obtained by using a set distance approach.
\begin{defn}[{\bf $h$-Convergent Subdivision}]\label{hconv}
A subdivision scheme is termed $h$-convergent if for arbitrary initial data $p^0$ there exists a set $p^\infty\subset \mathbb{R}^m$, such that
\begin{equation}\label{setlimit}
\lim_{k\to\infty}h(p^k,p^\infty)=0,
\end{equation}
where $h$ is the Euclidian Hausdorff metric on $\R^m$.
The set $p^\infty$ is termed the $h$-limit of the subdivision scheme.
\end{defn}

The relationship between
curves and surfaces generated by stationary subdivision algorithms and self-similar fractals generated by iterated function systems (IFSs) was first presented in \cite{SLG}. 
The work in \cite{LDV} establishes a relation between non-stationary subdivision with a mask of fixed support size and sequences of function systems (SFSs). The present paper is motivated by our goal to find the relationship between subdivision and fractals for an extended families of subdivision schemes
such as non-stationary schemes with growing mask size and non-uniform schemes. It turns out that two new structures of sequences of maps are needed here. The first involves an infinite binary tree of maps on a metric space and the second a sequence of metric spaces $\{(X_i,d_i)\}_{i\in\N}$ together with a ``staircase'' type sequence of maps $\{T_i\}_{i\in \N}$, $T_i:X_{i+1}\to X_i$. For both structures, we extend the Banach fixed point theorem.

In Section \ref{TM}, we study the structure of an infinite tree of maps (TMs). We consider convergence of backward trajectories along paths in the tree and introduce the notion of attractor of a tree. Staircase maps are analyzed in Section \ref{SM} and in Sections \ref{SMFS} and \ref{SubdSM} we demonstrate the application of staircase trajectories to SFSs related to non-stationary subdivision procedures, in particular that which generates the up-function \cite{DL}. In Section \ref{NUSTA}, staircase maps and trees of maps are applied to the analysis of linear non-uniform subdivision schemes.

The results in the present work are built upon the results in \cite{LDV} which we briefly review below.

\subsection{Sequences of Maps on $\{X,d\}$ and their Trajectories}

Let $(X,d)$ be a complete metric space. For a map $f: X \to X$, we define the Lipschitz constant associated with $f$ by
\[
\text{Lip}(f) = \sup_{x,y \in X, x \neq y} \frac{d\big(f(x),f(y)\big)}{d(x,y)}.
\]
A map $f$ is said to be Lipschitz if $\text{Lip}(f) < + \infty$ and a contraction if $\text{Lip}(f) < 1$.

Now, consider a sequence of continuous maps $\{T_i\}_{i \in \N}$, $T_i: X \to X$.

\begin{defn}[\textbf{Forward and Backward Trajectories}]

The forward and backward trajectories of $\{T_i\}_{i \in \N}$ in $X$, starting from $x\in X$, are
\begin{enumerate}
\item $ \Phi_k(x)=T_k \circ T_{k-1} \circ \dots \circ T_1(x)$, $k \in \mathbb{N}$,
respectively,
\item $ \Psi_k(x)=T_1 \circ T_2 \circ \dots \circ T_k(x)$, $k \in \mathbb{N}$.
\end{enumerate}
\end{defn}

In \cite{LDV}, the convergence of both types of trajectories is studied and the results are applied to the case where the maps are iterated function systems (IFSs) generating fractals. Such fractals are then related to limits of non-stationary subdivision procedures generating curves.
The following definition is presented in \cite{LDV} and is used in this paper.

\begin{defn}
Two sequences $\{x_i\}_{i \in \mathbb{N}}$ and $\{y_i\}_{i \in \mathbb{N}}$ in a complete metric space $(X,d)$ are said to be asymptotically equivalent if $d(x_i,y_i) \to 0$ as $i \to \infty$. We denote this relation by
\begin{equation}
\{x_i\}\sim \{y_i\}.
\end{equation}
Obviously $\sim$ is an equivalence relation.
\end{defn}
With this definition, we can formulate an important property of these trajectories.
\begin{prop}[\bf Asymptotic Equivalence of Trajectories]
\label{Equivalence}

Let $\{T_i\}_{i \in \mathbb{N}}$ be a sequence of transformations on $X$ where each $T_i$ is a Lipschitz map with Lipschitz constant $s_i$. If $\lim\limits_{ k \to \infty} \prod\limits_{i=1}^k s_i =0$ then
for any $x,y\in X$,
\begin{equation}
\begin{aligned}
\{\Phi_k(x)\}\sim \{\Phi_k(y)\},\\
\{\Psi_k(x)\}\sim \{\Psi_k(y)\}.
\end{aligned}
\end{equation}
\end{prop}

To present the results in \cite {LDV} concerning convergence of trajectories, we need the following definition.

\begin{defn}[\bf Invariant Domain of $\{T_i\}$]\label{def3}

We call $C\subseteq X$ an invariant domain of a sequence of transformations $\{T_i\}_{i \in \mathbb{N}}$ if
\begin{equation}\label{C}
\forall i\in \N\ ,\  \forall x\in C:\ T_i(x)\in C.
\end{equation}
\end{defn}

For the convergence of forward trajectories $\{\Phi_k(x)\}$, it is assumed in \cite{LDV} that the sequence of maps converges to a limit map: $\lim\limits_{i\to\infty}T_i=T$.

\begin{prop}[\bf Convergence of Forward Trajectories]
\label{forwardconvergence}

Let $\{T_i\}_{i \in \mathbb{N}}$ be a sequence of transformations on $X$ with a common compact invariant domain $C$. Let $T: C\to C$ be a Lipschitz map with Lipschitz constant $\mu<1$.
If
\begin{equation}\label{epssum}
\lim_{i\to\infty}\sup\limits_{x\in C}d(T_i(x),T(x))=0,
\end{equation}
then for any $x\in C$ the forward trajectory $\{\Phi_k(x)\}_{k \in \mathbb{N}}$ converges to the fixed point $p$ of $T$, $p=Tp$, namely,
\begin{equation}
\lim_{k\to\infty}d(\Phi_k(x),p)=0.
\end{equation}
\end{prop}

For the convergence of the backward trajectories $\{\Psi_k(x)\}$ the following result is presented in \cite{LDV}.

\begin{prop}[\bf Convergence of Backward Trajectories]
\label{BTp1}

Let $\{T_i\}_{i \in \mathbb{N}}$ be a sequence of transformations on $X$ with a common compact invariant domain $C$. Assume that each $T_i$ is a Lipschitz map with Lipschitz
constant $s_i$. If $\sum\limits_{k=1}^\infty \prod\limits_{i=1}^k s_i <\infty$, then the backward trajectories $\{\Psi_k(x)\}_{k \in \mathbb{N}}$ converge for all points $x\in C$ to a unique point in $C$.
\end{prop}

Proposition \ref{BTp1} is an example of an extension of the Banach fixed point theorem. Namely, for a given infinite iterative process, it states conditions that guarantee the existence of a basin of attraction from which the iterative process converges to a unique attractor. In this paper we present such extensions for several types of infinite iterative processes.

\section{Trees of Maps (TMs)}\label{TM}

Having in mind non-uniform subdivision schemes and possible applications to the generation of fractals which are non-uniform in space, we introduce the structure of a tree of maps. For expository purposes, we present the idea for the case of binary trees of maps. 

Let us denote by $B^{[k]}$ the set of binary codes of length $k$ from the alphabet $\{1,2\}$,
\begin{equation}\label{Ek2}
B^{[k]} :=\{\eta_k : \eta_k=(i_1 i_2 \ldots i_{k-1} i_k),\ i_j\in\{1,2\}\} = \{1,2\}^k.
\end{equation}
We denote the set of infinite binary codes by $B^{[\infty]}$,
\begin{equation}\label{Binf}
B^{[\infty]}:=\{\eta : \eta=(i_1 i_2 \ldots i_k\dots),\ i_j\in\{1,2\}\} = \{1,2\}^{\mathbb{N}}.
\end{equation}

In the following, we define operations on $B^{[k]}$ and on $B^{[\infty]}$.
\begin{defn}[\bf Prolongation and Truncation Operators]\label{previous}\hfill

For a finite code $\eta_k=(i_1 i_2 \ldots i_{k-1} i_k)\in B^{[k]}$, we define the { prolongation operator} $\pi_j: B^{[k]}\to B^{[k+1]}$ by
\[
\pi_j\,\eta_k := \eta_k j := (i_1 i_2 \ldots i_{k-1} i_k j),\ j=1,2.
\]

The {$\ell$-truncation} operator
$\tau_\ell: B^{[k]}\to B^{[\ell]}$ for $1\le \ell<k$, is defined by
\[
\tau_\ell\,\eta_k:= (i_1 i_2 \ldots i_\ell):= \eta_\ell.
\]
The last definition applies also to $\eta\in B^{[\infty]}$.
\end{defn}

Now, consider an infinite collection of continuous functions (maps) from $X$ to itself, 
\[
\{f_{\eta_k} : \eta_k\in B^{[k]}, k\in\mathbb{Z}_+\}.
\]
A function $f_{\eta_k}$ is also denoted by $f_{i_1 i_2 \ldots i_{k-1} i_k}$. We arrange the maps in a tree structure, as depicted below,
\tikzstyle{level 1}=[level distance=2cm, sibling distance=2cm,->]
\tikzstyle{level 2}=[level distance=2cm, sibling distance=1cm,->]
\tikzstyle{level 3}=[level distance=2cm, sibling distance=1cm,->]
 
\tikzstyle{bag} = [text width=2.5em, text centered]
\tikzstyle{end} = []
 
\begin{tikzpicture}[grow=right, sloped]
\node[bag] {$I$}
     child {
         node[bag] {$f_{2}$}       
             child {
                 node[bag]
                     {$f_{22}\ .\ .\ .$}
                 edge from parent
                 node[above] {}
                 node[below]  {}
             }
             child {
                 node[bag]
                     {$f_{21}\ .\ .\ .$}
                 edge from parent
                 node[above] {}
                 node[below]  {}
             }
             edge from parent 
             node[above] {}
             node[below]  {}
     }
     child {
         node[bag] {$f_{1}$}        
         child {
                 node[bag]
                     {$f_{12}$}
                      child {
                      node[bag]
                     {$f_{122}\ .\ .\ .$} {}
                     child {node {$$} edge from parent[draw=none]}
                     child {node {$\ .\ .\ .\, f_{\eta_k}$}
                      child {
                 node[bag]
                     {$f_{\eta_k2}\ .\ .\ .$}
                 edge from parent
                 node[above] {}
                 node[below]  {}
             }
             child {
                 node[bag]
                     {$f_{\eta_k1}\ .\ .\ .$}
                 edge from parent
                 node[above] {}
                 node[below]  {}
             }
             edge from parent[draw=none]}
                 edge from parent
                 node[above] {}
                 node[below]  {}
             }%
              child {
                 node[bag]
                     {$f_{121}\ .\ .\ .$} {}
                 edge from parent
                 node[above] {}
                 node[below]  {}
             }%
                 edge from parent
                 node[above] {}
                 node[below]  {}
             }
             child {
                 node[bag]
                     {$f_{11}\ .\ .\ .$}
                 edge from parent
                 node[above] {}
                 node[below]  {}
             }
         edge from parent         
             node[above] {}
             node[below]  {}
     };
\end{tikzpicture}
where the functions corresponding to the two branches of the same code $f_{\eta_{k1}}$ and $f_{\eta_{k2}}$ are different.

\subsection{Attractor of a Tree of Maps}

In the following, we define the notion of the attractor of a tree of maps. We present sufficient conditions for the convergence of an infinite process on the tree to a unique attractor $U_{TM}\subset C\subset X$ where $C$ is a common invariant domain of all the maps $\{f_{\eta_k} : \eta_k\in B^{[k]}, k\in\N\}$ in the tree.

\begin{defn}[\bf Paths in a Tree]\label{paths}

Each infinite code $\eta\in B^{[\infty]}$ defines
a path in the tree which is an infinite sequence of codes $P_\eta=\{\eta_1,\eta_2,\ldots,\eta_k,\ldots\}$ such that $\eta_\ell=\tau_\ell\eta$ for $\ell\in \N$.
\end{defn}

A path $P_\eta$ defines a sequence of maps from $X$ to itself, $\{f_{\eta_k}\}_{k\in\N}$. 
Therefore, the proof of the next proposition is a direct consequence of Proposition \ref{BTp1} for the convergence of backward trajectories.

\begin{prop}[\bf Convergence along a Path $P_\eta$]\label{PTFS}

Let $P_\eta$ be a path in a tree $T$ of maps and let $\{f_{\eta_k}\}_{\eta_k\in P_\eta}$ be the sequence of maps along this path with a common compact invariant domain $C_{P_\eta}$ and with associated Lipschitz constants $\{s_{\eta_k}\}_{\eta_k\in P_\eta}$.
If $\sum\limits_{k=1}^\infty \prod\limits_{i=1}^k s_{\eta_i}<\infty$ then
\begin{equation}\label{limitpath}
\gamma(\eta):=\lim_{k\to\infty} f_{\tau_1\eta}\circ \cdot\cdot\cdot\circ f_{\tau_{k-1}\eta}\circ  f_{\tau_k\eta}(x)
\end{equation}
exists for all $x\in C_{P_\eta}$, and is the same element in $C_{P_\eta}$.
\end{prop}

\begin{prop}[\bf Convergence to $\boldsymbol{U}_{TM}$]\label{CTFS}

If the conditions of Proposition \ref{PTFS} are satisfied for all paths in T with the same invariant domain $C$, then the following set $U_{TM}\subset C$ is uniquely defined for all $x\in C$:
\begin{align}\label{UTM}
U_{TM} &=\bigcup\limits_{\eta\in B^{[\infty]}}\gamma(\eta)
& = \bigcup\limits_{\eta\in B^{[\infty]}} \lim_{k\to\infty} f_{\tau_1\eta}\circ \cdot\cdot\cdot\circ f_{\tau_{k-1}\eta}\circ  f_{\tau_k\eta}(x).
\end{align}
\end{prop}

\begin{defn}[\bf Convergent TMs]\label{convergentTM}

Let the TM be such that for every path in it, the limit (\ref{limitpath}) exists and is the same for all $x\in C\subset X$. Then we call the TM \emph{convergent} and term the set 
$U_{TM}=\bigcup\limits_{\eta\in B^{[\infty]}}\gamma(\eta)$ the \emph{attractor of TM}.
\end{defn}

\subsection{Changing the Order of $\ \bigcup$ and $\lim$ in (\ref{UTM}).}

In order to ensure a practical computation of $U_{TM}$, we should consider changing the order of the operations $\ \bigcup$ and $\lim$ in (\ref{UTM}). Changing the order requires a stronger assumption on the convergence along the paths of TM, namely, a uniform convergence, uniform on all paths in the tree. The objects we deal with are nonvoid compact subsets of $X$. Hence, the convergence we ask for is in ${\bH}(X)$, the collection of all nonvoid compact subsets of $X$. Note that ${\bH}(X)$ becomes a complete metric space when endowed with the Hausdorff metric $h$ \cite{B2}.

\begin{prop}[\bf Uniform onvergence to $\boldsymbol{U}_{TM}$]\label{Uconv}

Let us ssume that all the maps in $T$ share the same common invariant domain $C$.
We further assume that for all paths in $T$
\begin{equation}\label{uniformbound}
\prod\limits_{i=1}^k s_{\tau_i\eta}\le \delta_k,\ \ \ \forall \eta\in B^{[\infty]},
\end{equation}
where $\sum\limits_{k=1}^{\infty}\delta_k<\infty$. Then,
\begin{equation}\label{UTMek}
U_{TM}=\lim_{k\to\infty}\bigcup\limits_{\eta\in B^{[\infty]}} f_{\tau_1\eta}\circ \cdot\cdot\cdot\circ f_{\tau_{k-1}\eta}\circ f_{\tau_{k}\eta}(x),
\end{equation}
where the convergence is with respect to the Hausdorff metric $h$.
\end{prop}

\begin{proof}
The assumption (\ref{uniformbound}) implies that
\begin{equation}\label{uniformbound2}
\sum_{k=m}^\infty\prod\limits_{i=1}^k s_{\tau_i\eta}\le e_m,\ \ \ \forall \eta\in B^{[\infty]},
\end{equation}
where $\lim\limits_{m\to\infty}e_m=0$.
For $x\in C$ and $k\in\mathbb{N}$, let
\begin{equation}\label{partialkpath}
\gamma_k(\eta)(x):= f_{\tau_1\eta}\circ \cdot\cdot\cdot\circ f_{\tau_{k-1}\eta}\circ  f_{\tau_k\eta}(x).
\end{equation}
Following the proof of convergence of backward trajectories in \cite{LDV}, we note that for every  $x\in X$ and $\eta\in B^{\infty}$, $\{\gamma_k(\eta)(x)\}$ is a Cauchy sequence in $(X,d)$ and
\begin{equation}\label{leem}
d(\gamma_m(\eta)(x),\gamma_{m+\ell}(\eta)(x))\le Ee_m,\ \ \forall \ell\in\mathbb{N},
\end{equation}
where $E=\diam(C)$.

Since $\lim\limits_{k\to\infty}\gamma_k(\eta)(x)=\gamma(\eta)$, it follows that 
\begin{equation}\label{leem2}
d(\gamma_m(\eta)(x),\gamma (\eta))\le Ee_m.
\end{equation}
Hence,
$$h\left(\bigcup\limits_{\eta\in B^{[\infty]}}\gamma_m(\eta)(x),\bigcup\limits_{\eta\in B^{[\infty]}}\gamma(\eta)\right)\le Ee_m,$$
which implies the convergence of (\ref{UTMek}) with respect to $h$.
\end{proof}

\subsubsection{Self-Referential Property}\hfill

A convergent TM with maps $\{f_{\eta_k} : \eta_k\in B^{[k]}, k\in\N\}$, defines two subtrees, TM1 with the maps $\{f_{1\eta_k} : \eta_k\in B^{[k]}, k\in\N\}$ and TM2 with the maps $\{f_{2\eta_k} : \eta_k\in B^{[k]}, k\in\N\}$. Assume that TM1 has attractor $U_{TM1}$ and TM2 has attractor $U_{TM2}$. Then these two attractors are related to the attractor $U_{TM}$ of TM by 
\begin{equation}\label{srTM}
U_{TM}=f_1(U_{TM1})\cup f_2(U_{TM2}).
\end{equation}
\begin{rem}[General TMs]\label{GenTrees}

The above results are presented for binary TMs. The case of general TMs involves more complicated indexing but can be treated in exactly the same manner. For general TMs, the union operation in the definition of the limit set $U_{TM}$ in (\ref{UTM}) is over all the paths in the tree.
\end{rem}

\subsubsection{Code Dependence and Location Dependence}\label{explainTFS}\hfill

Within a TM, we have a collection of maps which are code dependent. We explain below that code dependence can imply location dependence. That is, different paths lead to different locations in the attractor $U_{TM}$.

Let $P_\eta$ and $P_\nu$ be two paths in a convergent TM with $\eta=(i_1i_2i_3\dots)$ and $\nu=(j_1j_2j_3\dots)$. Let $\gamma(\eta)$, respectively, $\gamma(\nu)$ be the two points in $X$ generated by the limits (\ref{limitpath}). Endowed with the Fr\'echet metric $d_F:B^{[\infty]}\times B^{[\infty]} \to\R$, 
\[
d_F (P_\eta, P_\nu):= \sum_{n=1}^\infty \frac{|i_n - j_n|}{3^n}, 
\]
$(B^{[\infty]}, d_F)$ becomes a compact metric space. It is known (cf. for instance, \cite{B1,PRM}) that there exists a continuous surjection $\gamma$ from code space $B^{[\infty]}$ to the attractor of an IFS. In the TM setting, the mapping $\gamma (\eta)$ is also a continuous map from $B^{[\infty]}$ to the attractor $U_{TM}$, as shown below. 

\begin{prop}\label{contmap}

Under the conditions of Proposition \ref{CTFS}, the mapping $\gamma (\eta)$ is a continuous surjection from $B^{[\infty]}$ onto $U_{TM}$. 
\end{prop}

\begin{proof}
To prove continuity at $P_\eta$, let us assume that $d_F(P_\eta, P_\nu)<\delta$. It follows that $i_m=j_m$ for $1\le m\le \ell(\delta)$ where $\ell(\delta)=[-\log_3(\delta)]$.
Using the expression in (\ref{limitpath}), we examine the distance
\begin{equation}
\epsilon_k :=  d( f_{\eta_1}\circ \cdot\cdot\cdot\circ f_{\eta_{k-1}}\circ  f_{\eta_k}(x), f_{\nu_1}\circ \cdot\cdot\cdot\circ f_{\nu_{k-1}}\circ  f_{\nu_k}(x)),\ \ x\in C,\ \  k>\ell(\delta).
\end{equation}
Since $\eta_m=\tau_m$ for $1\le m\le \ell(\delta)$, it follows by recursive application of 
$$d(f_{\eta_m}(a),f_{\eta_m}(b))\le s_{\eta_m}d(a,b),$$
that
$$
\epsilon_k\le \prod_{m=1}^{\ell(\delta)}s_{\eta_m}\cdot d( f_{\eta_{\ell(\delta)+1}}\circ \cdot\cdot\cdot\circ  f_{\eta_k}(x), f_{\nu_{\ell(\delta)+1}}\circ \cdot\cdot\cdot\circ  f_{\nu_k}(x)).
$$
Hence,
$$
\epsilon_k\le D\prod_{m=1}^{\ell(\delta)}s_{\eta_m},\ \ \forall k>\ell(\delta),
$$
where $D$ is the diameter of $C$. Consequently,
$$
d(\gamma(\eta),\gamma(\nu))\le D\prod_{m=1}^{\ell(\delta)}s_{\eta_m}.
$$
The proof follows by observing that as $\delta\to 0$, $\ell(\delta)\to \infty$ and $\prod\limits_{m=1}^{\ell(\delta)}s_{\eta_m}\to 0$.
\end{proof}

The continuity of $\gamma$ implies that the code dependence of the functions in the tree means location dependence. For more details about the relation between codes and points in the attractor of an iterated function system, see e.g. \cite{B1,PRM}.

In the next subsection we apply the above results to trees of function systems consisting of pairs of maps each.

\subsection{Trees of Function Systems (TFSs)}

Recall that an iterated function system (IFS) is a pair consisting of a complete metric space $(X,d)$ and a finite family of continuous maps $f_i: X \to X$, $i \in \{1,2,\dots,n\}$. We denote such an IFS by $\mathcal{F}=\{X; f_i: i=1,2,\dots, n\}$.  If the $f_i$ are contraction maps, the IFS is called contractive. The contraction constant of $\mathcal{F}$ is $L_{\mathcal{F}}=\max\limits_{i=1,2,\dots, n} \text{Lip}(f_i)$.

Consider the set-valued mapping $\mathcal{F}: \mathbb{H}(X)\to \mathbb{H}(X)$,
\begin{equation}\label{FX}
 \mathcal{F}(B) := \bigcup\limits_{f \in \mathcal{F}} f(B),\ \ B\in \mathbb{H}(X),
\end{equation}
where $f(B):= \big\{f(b): b \in B \big\}$.
It is well known that for a contractive IFS,  $\mathcal{F}$ is a contraction in $(\mathbb{H}(X),h)$ with contraction constant $L_{\mathcal{F}}=\max\limits_{i=1,2,\dots, n} \text{Lip}(f_i)$ \cite{B2}. Therefore, by the Banach contraction principle $\mathcal{F}$  has a unique fixed point, denoted here by $U_{IFS}$, called the attractor of the IFS.

Here, we consider a binary IFS $\mathcal{F}=\{f_1,f_2\}$. If both functions are contractions on $X$, the attractor of the IFS is given by
\begin{equation}\label{UIFS}
U_{IFS}=\lim_{k\to\infty}\bigcup\limits_{(i_1 i_2\ldots i_k)\in B^{[k]}} f_{i_1}\circ f_{i_2}\circ \cdot\cdot\cdot\circ f_{i_k}(A),
\end{equation}
where $A$ is any non-void compact subset of $X$. By Proposition \ref{Uconv},
the convergence in (\ref{UIFS}) is ensured.

A binary tree of maps may be considered as a tree of function systems. Each pair of functions ($f_{\eta_k1}$, $f_{\eta_k2}$) defines a function system, which we denote by $\mathcal{F}_{\eta_k}$.
Consequently, the above tree of maps induces an infinite binary tree of function systems (TFS) $\{\mathcal{F}_{\eta_k}\}_{\eta_k\in B^{[k]},\,k\in \N}$ as follows.
For $k=0$, set
$$\mathcal{F}_0=\{f_{1},f_{2}\},$$
and for $\eta_k\in B^{[k]}$, $k\in \N$,
$$ \mathcal{F}_{\eta_k} = \{f_{\eta_k 1},f_{\eta_k 2}\}.$$

Following backward paths on the tree of function systems, the natural generalization of the fractal $U_{IFS}$ in (\ref{UIFS}) is
\begin{equation}\label{UTFS2}
U_{TFS}=\lim_{k\to\infty}\bigcup\limits_{\eta_k\in B^{[k]}} f_{\tau_1\eta_k}\circ f_{\tau_2\eta_k} \cdot\cdot\cdot\circ f_{\eta_k}(A).
\end{equation}
To ensure convergence in (\ref{UTFS2}), with respect to $h$, we assume here that the condition (\ref{uniformbound2}) in Proposition \ref{Uconv} is satisfied, with $s_{\eta_k}=\Lip(\mathcal{F}_{\eta_k})$.

If all the function systems are the same IFS, we have $U_{TFS}=U_{IFS}$. If the function systems 
$\{\mathcal{F}_{\eta_k}\}$ are all the same for a fixed $k$, we retrieve the case of sequences of function systems discussed in \cite{LDV}. In the following example we present a tree of function systems and its attractor.

\begin{ex}\label{TFSexample}

To each code $\eta_k=(i_1i_2...i_k)$ we attach a number 
\[
t(\eta_k) := \sum_{j=1}^k (i_j-1)2^{-j}\in [0,1].
\]
We define the function systems in the TFS by the following maps on $X=\mathbb{R}^2$:
\[
f_{\eta_k1}\left(\begin{pmatrix} x_1 \\ x_2\end{pmatrix}\right)=\begin{pmatrix} 0.4 & 0.4\\ -0.5 & 0.3 \end{pmatrix} \begin{pmatrix} x_1 \\ x_2\end{pmatrix} + \begin{pmatrix} 4 \\ 4\end{pmatrix},
\]
and

\[
f_{\eta_k2}\left(\begin{pmatrix} x_1 \\ x_2\end{pmatrix}\right)=0.8\begin{pmatrix} \cos \theta_k & \sin\theta_k\\ -\sin\theta_k & \cos\theta_k \end{pmatrix} \begin{pmatrix} x_1 \\ x_2\end{pmatrix} + \begin{pmatrix} -2 \\ 0\end{pmatrix},
\]
where $\theta_k :=3t(\eta_k)+1$, $k\in \N$.

All of the above maps are contractive and thus the tree of maps is convergent by Proposition \ref{CTFS}. In Figure \ref{Tfractal2d4} below, we depict its attractor. The noticeable property of this attractor is that is has different structures at different locations. The challenge is to investigate how to design desirable structures at specific locations.

\begin{figure}[!ht]
    \includegraphics[width=4in]{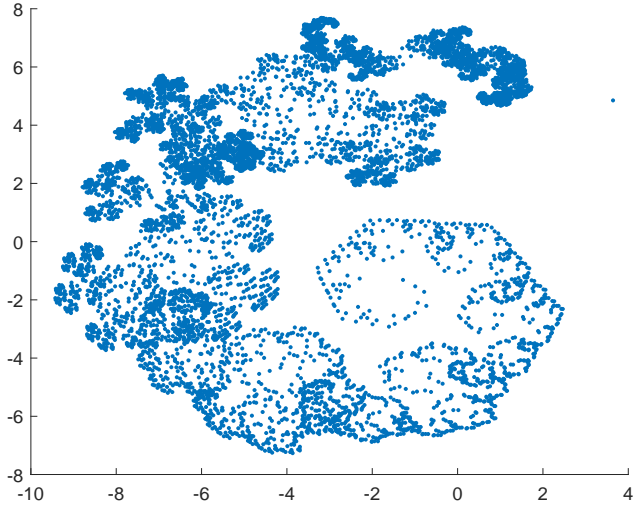}
    \caption{An attractor $U_{TFS}$ showing different structures at different locations.}
    \label{Tfractal2d4}
\end{figure}
\end{ex}

\section{Staircase Maps on a Sequence of Metric Spaces and their Trajectories}\label{SM}

Let $(X,d_X)$ and $(Y,d_Y)$ be two complete metric spaces and let
$T: X \to Y$ be a continuous map. Consider a non-void compact set $C\subset X$.

\begin{defn}\label{LipC}

The Lipschitz constant associated with $f$ on $C$ is defined by
$$\text{Lip}_C(T) := \sup_{a,b \in C, a \neq b} \frac{d_Y\big(T(a),T(b)\big)}{d_X(a,b)}.$$
A map $T$ is said to be Lipschitz map with respect to $C$ if $\text{Lip}_C(T) < + \infty$ and a contraction with respect to $C$ if $\text{Lip}_C(T) < 1$.
\end{defn}

Let us consider an infinite sequence of complete metric spaces $\{(X_i,d_i)\}_{i\in\N_0}$ and an associated sequence of continuous maps $\{T_i\}_{i\in\N}$,

\begin{equation}\label{Ti}
T_i:X_{i}\to X_{i-1}, \ \ \ i\in\N.
\end{equation}

Correspondingly, we assume the existence of non-void compact subsets $\{C_i\}_{i\in\N_0}$, $C_i\subset X_i$, of bounded diameters, $\diam (C_i)\le D$, such that

\begin{equation}\label{TionCi}
T_i:C_{i}\to C_{i-1}, \ \ \ i\in\N.
\end{equation}

We further assume that each $T_i$ is Lipschitz with respect to $C_{i}$, ${i\in\N}$. We denote its Lipschitz constant by $s_i$:

\begin{equation}\label{LipTi}
s_i :=\text{Lip}_{C_{i}}(T_i), \ \ \ i\in\N.
\end{equation}

In order to consider backward trajectories in the present general setting of maps between different spaces, we need to modify their definition.
We relate a trajectory of $\{T_i\}_{i\in\N}$ to a base sequence of points
\begin{equation}\label{barx}
\bar{x}=\{x_{i}\}_{i\in\N},\ \ \ \ x_i\in C_i.
\end{equation}
We refer to such a sequence as a \emph{base sequence for a backward staircase trajectory}.

\begin{defn}[\bf Backward Staircase Trajectories of $\{T_i\}$ with respect to $\bar{x}$]\label{defbt}

Under (\ref{Ti}) and (\ref{TionCi}), a trajectory of the sequence of maps $\{T_i\}_{i\in\N}$ with respect to a sequence of base points $\bar{x}$ as in (\ref{barx}) is the sequence of points $\{p_k(\bar{x})\}_{k\in\N}\subset X_0$ defined by
\begin{equation}\label{pk}
p_k(\bar{x})=T_1 \circ T_2 \circ \dots \circ T_k(x_{k}),\ \ k \in \mathbb{N}.
\end{equation}
\end{defn}


\begin{prop}[\bf Asymptotic Equivalence of Backward Staircase Trajectories]\label{Equivalence2}

Let $\{(X_i,d_i)\}_{i\in\N_0}$, $\{T_i\}_{i\in\N}$ , $\{C_i\}_{i\in\N_0}$, and $\{s_i\}_{i\in\N}$ be as above.
If $\ \lim\limits_{ k \to \infty} \prod\limits_{i=1}^k s_i =0$, then for any two base sequences $\bar{x}$ and $\bar{y}$,
\begin{equation}
\lim_{k\to\infty}d_0(p_k(\bar{x}),p_k(\bar{y}))=0.
\end{equation}
\end{prop}

The proof of this result is similar to the proof in \cite{LDV} for backward trajectories. We present it here for the sake of completeness.

\begin{proof}

Let $\bar{x},\bar{y}$ be two base sequences and consider the backward staircase trajectories defined in (\ref{pk}). Using the fact that $T_i$ is a Lipschitz map with Lipschitz constant $s_i$ and that $\diam_i(C_i)\le D$ for $i\in\N$, we have that
\begin{equation}\label{prodsi}
\begin{aligned}
d_0(p_k(\bar{x}),p_k(\bar{y}))\ & \le\ s_1 d_1(T_2 \circ T_3 \circ \dots \circ T_k(x_{k}),T_2 \circ T_3 \circ \dots \circ T_k(y_{k})) \\
& \le s_1s_2d_2((T_3 \circ T_4 \circ \dots \circ T_k(x_{k}),T_3 \circ T_4 \circ \dots \circ T_k(y_{k})) \le\ldots \\
& \le \left(\prod_{i=1}^k s_i\right) d_{k}(x_{k},y_{k})\ \le D\left(\prod_{i=1}^k s_i\right),
\end{aligned}
\end{equation}
from which the result follows.
\end{proof}

Note that the condition $\ \lim\limits_{ k \to \infty} \prod\limits_{i=0}^k s_i =0$ certainly holds if $s_i<s<1$, for all $i\in\N$.

\begin{prop}[\bf Convergence of Backward Staircase Trajectories]\label{STp2}
Let 

\noindent
$\{(X_i,d_i)\}_{i\in\N_0}$, $\{T_i\}_{i\in\N}$, $\{C_i\}_{i\in\N_0}$, and $\{s_i\}_{i\in\N}$ be as above.
If $\sum\limits_{k=1}^\infty \prod\limits_{i=1}^k s_i <\infty$, then for all base sequences $\bar{x}$ the staircase trajectories
$\{p_k(\bar{x})\}$ defined by (\ref{pk}) converge to a unique limit in $C_0$.
\end{prop}
\begin{proof}

By (\ref{pk}), the relations in (\ref{prodsi}), and the assumption $\diam (C_i)\le D$ for $i\in\N$, we have that

\begin{equation}\label{prodsi2}
\begin{aligned}
d_0(p_{k+1}(\bar{x}),p_k(\bar{x})) & = d_0(T_1 \circ T_2 \circ \dots \circ T_{k+1}(x_{k+1}),T_1 \circ T_2 \circ \dots \circ T_k(x_{k})) \\
& \le \left(\prod_{i=1}^k s_i\right) d_{k}(T_{k+1}(x_{k+1}),x_{k})\ \le D\left(\prod_{i=1}^k s_i\right).
\end{aligned}
\end{equation}

For $m,k \in \mathbb{N}$ with $m>k$, we obtain in view of (\ref{prodsi2}),
\begin{align}\label{e1}
d_0\big(p_m(\bar{x}),p_k(\bar{x})\big) \le &~ d_0\big(p_m(\bar{x}),p_{m-1}(\bar{x})\big) + \dots\\
& + d_0\big(p_{k+2}(\bar{x}),p_{k+1}(\bar{x}) \big)+d_0\big(p_{k+1}(\bar{x}),p_k(\bar{x}) \big)\nonumber\\
\le &~ D\prod_{i=1}^{m-1}s_i+\dots + D\prod_{i=1}^{k+1}s_i+ D\prod_{i=1}^{k}s_i = D\left(\sum_{j=k}^{m-1} \prod_{i=1}^j s_i\right).
\end{align}
Since $\sum\limits_{k=1}^\infty \prod\limits_{i=1}^k s_i < \infty$,  Eq. (\ref{e1}) asserts that
$d_0\big(p_m(\bar{x}),p_k(\bar{x})\big) \to 0$ as $ k \to \infty$.
That is, $\{p_k(\bar{x}) \}_{k \in \mathbb{N}}$ is a Cauchy sequence in $C_0$ and
due to the completeness of $(X_0,d_0)$ and the compactness of $C_0$, it is convergent to a limit in $C_0$.
The uniqueness of the limit is derived by the equivalence of all trajectories as proved in Proposition \ref{Equivalence2}.
\end{proof}

\begin{rem}[Invariance to Scaling of the Metrics $\{d_i\}$]\label{freedom}

In view of Definition \ref{LipC}, the Lipschitz constant of a mapping from one metric space to another, it is clear that by a proper scaling of the metrics $\{d_i\}$ we can make all the maps $\{T_i\}$ contractive. By scaling the metrics we mean replacing $\{d_i\}$ with $\{\tilde{d}_i\}$, where
\begin{equation}\label{scaling}
\tilde{d}_i(x,y)=\alpha_id_i(x,y) \ \ \forall x,y \in X_i,\ \ \alpha_i>0.
\end{equation}
Such a scaling implies new Lipschitz constants
\begin{equation}\label{newsi}
\tilde{s}_i=\frac{\alpha_{i-1}}{\alpha_i}s_i.
\end{equation}
Choosing $\alpha_0=1$ we obtain

\begin{equation}\label{prodsi4}
 \prod_{i=1}^k \tilde{s}_i =\frac{1}{\alpha_{k}} \prod_{i=1}^k s_i.
\end{equation} 

Let us discuss the implications of the scaling of the metrics. To make the discussion more concrete, consider the case $X_i=\mathbb{R}^i$. A natural possibility is to use the same metric for all these spaces. For example, the metric induced by the maximum norm. If the conditions for asymptotic equivalence or for convergence are not satisfied for these metrics, we can scale the metrics up ($\alpha_i>1$)  in the case of backward trajectories. As we explain below, convergence can be obtained by scaling the metrics.

In view of the proof of Proposition \ref{STp2}, the critical quantity for convergence is the expression
\[
\left(\prod_{i=1}^k s_i\right) d_{k}(T_{k+1}(x_{k+1}),x_{k})
\] 
in equation (\ref{prodsi2}).
Because of (\ref{scaling}) and (\ref{prodsi}), we find that this term is invariant under the scaling of metrics:
\begin{equation}
\left(\prod_{i=1}^k \tilde{s}_i\right) \tilde{d}_{k}(T_{k+1}(x_{k+1}),x_{k})=\left(\prod_{i=1}^k s_i\right) d_{k}(T_{k+1}(x_{k+1}),x_{k}).
\end{equation}
We see that by our method of proof nothing can be gained by scaling the metrics. 
\end{rem}

\begin{rem}[Grouping]\label{grouping}

Sometimes the conditions for the convergence of backward trajectories in Proposition \ref{STp2} are not satisfied but they still may converge. One way of relaxing the conditions is to form groups of maps. Consider the following grouping of $\ell$ maps:
\begin{equation}\label{groupm}
\begin{aligned}
&G^{[\ell]}_j:X_{j\ell}\to X_{(j-1)\ell}, \ \ \ j\in\N, \\
&G^{[\ell]}_j :=T_{(j-1)\ell+1}\circ T_{(j-1)\ell+2} \circ ... \circ T_{j\ell-1} \circ T_{j\ell}.
\end{aligned}
\end{equation}
If for some $\ell$ the conditions for convergence of backward trajectories of $\{G^{[\ell]}_j\}$ are fulfilled, then the backward trajectories of $\{T_i\}$ will converge for some base sequences.
\end{rem}

\begin{cor}[\bf Partial Backward Staircase Trajectories]\label{Pst}

Let us assume that the backward staircase trajectory of $\Sigma_1 := \{T_i\}_{i\in\N}$
converge to a unique limit for all possible base sequences. Let us denote this limit by $\tau_0$.
Consider the partial sequences $\Sigma_m=\{T_i\}_{i=m}^\infty$, $m>1$ and assume that the staircase trajectories of each $\Sigma_m$ converge to a unique limit $\tau_{m-1}$, for all possible base sequences. Then,
\begin{equation}\label{sigma}
\tau_{m-1}=T_m(\tau_{m}),\quad  m\in\N.
\end{equation}
\end{cor}
\begin{rem}
As we extended the notion of backward trajectories from sequences of maps in a metric space to staircase trajectories, it is also possible to extend the notion of trees of maps on $X$ to a tree of staircase maps.
\end{rem}

\section{Sequences of Function Systems on a Sequence of Metric Spaces}\label{SMFS}

Consider an infinite sequence of complete metric spaces, $\{(X_i,d_i)\}_{i\in\N_0}$
and a related sequence of function systems $\{\mathcal{F}_i\}_{i\in \mathbb{N}}$ defined by
\[
\mathcal{F}_i = \left\{f_{1,i}, f_{2,i}, \dots, f_{n_i,i} \right \},
\]
where $f_{r,i}: X_{i} \to X_{i-1}$, $r=1,2,...,n_i$, are continuous maps. 
The associated set-valued maps are given by
\[
\mathcal{F}_i: \mathbb{H}(X_{i}) \to \mathbb{H}(X_{i-1}); \quad \mathcal{F}_i(A)= \bigcup\limits_{r=1}^{n_i} f_{r,i}(A).
\]

We now assume the existence of non-void compact subsets $\{C_i\}_{i\in\N_0}$, $C_i\subset X_i$ of bounded diameters, $\diam (C_i)\le D$, such that
\begin{equation}\label{fionCi}
f_{r,i}:C_{i}\to C_{i-1},\quad r=1,...,n_i, \ \ \ i\in\N .
\end{equation}

We further assume that each $f_{r,i}$ is Lipschitz with respect to $C_{i}$, $i\in\N$, and we denote its Lipschitz constant by
\begin{equation}\label{LipTi2}
s_{r,i}=\text{Lip}_{C_{i}}(f_{r,i}).
\end{equation}
The contraction factor of $\mathcal{F}_i$ in $(\mathbb{H}(X_i),h)$
is $L_{\mathcal{F}_i}=\max\limits_{r=1,2,\dots, n_i} s_{r,i} =: s_i$ \cite{LDV}.
Let
\begin{equation}\label{barA}
\bar{A}=\{A_{i}\}_{i\in\N},\quad A_i\subset C_i
\end{equation}
be a base sequence.

The following results follow directly from Proposition \ref{STp2} and Corollary \ref{Pst}.

\begin{cor}[\bf Convergence of Staircase SFS]\label{STSFS}

Consider the above staircase structure subject to the above assumptions.
Then, if 
\[
\sum_{k=1}^\infty \prod_{i=1}^k s_i <\infty,
\]
the staircase trajectory of $\Sigma_1=\{\mathcal{F}_i\}_{i\in\N}$,
\begin{equation}\label{stsfs}
p_k(\bar{A})=\mathcal{F}_1 \circ \mathcal{F}_{2} \circ \dots \circ \mathcal{F}_k(A_{k}),\ \ k \in \mathbb{N},
\end{equation}
converges for any base sequence $\bar{A}$ to a unique set (attractor) $P_0\subseteq C_0$.
\end{cor}

\begin{cor}[\bf Staircase Self-referential Property]\label{PSTSFS}

Under the conditions of Corollary \ref{STSFS}, the staircase trajectories of $\Sigma_m=\{\mathcal{F}_i\}_{i=m}^\infty$, namely, the sequence of sets
\begin{equation}\label{stsfs2}
p_k(\bar{A})=\mathcal{F}_m \circ \mathcal{F}_{m+1} \circ \dots \circ \mathcal{F}_k(A_{k}),\ \ k\ge m,
\end{equation}
converge for any base sequence to a unique set (attractor) $P_{m-1}\subseteq C_{m-1}$.
Furthermore,
\begin{equation}\label{sigma1}
P_{m-1}=\mathcal{F}_m(P_{m})= \bigcup\limits_{r=1}^{n_m} f_{r,m}(P_{m}),\quad  m \in \N.
\end{equation}

\end{cor}

\begin{rem}
Equation \eqref{sigma1} extends the self-referential property of fractals generated by IFSs to the setting of SFSs \cite{B2}.
\end{rem}

\section{From Subdivision to Staircase SFS}\label{SubdSM}

In this section, we first review the relation between non-stationary subdivision with a mask of fixed size and backward trajectories of a SFS as presented in \cite{LDV}. Then we suggest a framework for dealing with non-stationary subdivision with masks of increasing size.

Starting with a set of control points $p^{0}=\{p_j^0\in \mathbb{R}^m,\ \ j \in \mathbb{Z}\}$ at level $0$,
a non-stationary binary subdivision processes can be written in matrix form as
\begin{equation}\label{short3}
\bar{p}^{k+1}=S^{[k]}\bar{p}^k, \ \ k=1,2,...\ ,
\end{equation}
where $\bar{p}^k$ is the infinite matrix whose rows are the points at level $k$ of the subdivision,
and each $S^{[k]} := S_{a^{[k]}}$ is a ``two-slanted" infinite matrix.
Namely, $(S_{a^{[k]}})_{i,j}=a^{[k]}_{i-2j}$,
where $a^{[k]}$ is a finite sequence of reals termed the mask of the subdivision process.
The refinement rule (5.1) can be written as
$p^{k+1}_i=\sum_j a^{[k]}_{i-2j}p^k_j.$ 

As demonstrated in \cite{DL1}, non-stationary subdivision processes can generate interesting limits which cannot be generated by stationary schemes ($S^{[k]}=S$), e.g., exponential splines. Interpolatory non-stationary subdivision schemes can generate new types of compactly supported orthogonal wavelets as shown in \cite{DKLR}.

\subsection{Non-Stationary Subdivision with Masks of Fixed Size}

It is assumed in \cite{LDV} that the supports of the masks $a^{[k]}$, $|\sigma(a^{[k]})|$, are of the same size which is smaller than the number of initial control points.
For a given finite set of control points, $p^0=\{p^0_j\}_{j=1}^n$, \cite{LDV} defines for each $k$ the two square $n\times n$ sub-matrices of each $S^{[k]}$, $S^{[k]}_1$ and $S^{[k]}_2$, in the same way as suggested for a stationary scheme in \cite{SLG}. Out of the points generated at the first level of the subdivision, the
first $n$ points are the rows $S^{[1]}_1\bar{p}^0$ and the last $n$ points are the rows of $S^{[1]}_2\bar{p}^0$. At the second level we apply
$S^{[2]}_1$ and $S^{[2]}_2$ to the two resulting vectors of points and so on.
The set of all the points generated at level $k$ of the subdivision process is given by
\begin{equation}\label{nspk}
p^k=\bigcup_{i_1,i_2,...,i_k\in \{1,2\}} \rows\big(S^{[k]}_{i_k},...,S^{[2]}_{i_2}S^{[1]}_{i_1}\bar{p}^0\big)\ .
\end{equation}
Here, $\rows(\bar{p})$ denotes the set of points comprised of the rows of $\bar{p}$. If the subdivision is convergent \cite{LDV},
\begin{equation}\label{pktopinfty}
p^k\to p^\infty\quad \text{as $k\to\infty$},
\end{equation}
where $p^\infty$ is the set of points defined by the non-stationary subdivision process starting with $p^0$.

As shown in \cite{LDV}, the same limit is obtained by backward trajectories of a related SFS
$\{\mathcal{F}_k\}$, where
$\mathcal{F}_k=\big\{X; f_{1,k}, f_{2,k} \big \}$
with level dependent maps
\begin{equation}\label{f1f2k}
f_{r,k}(A)=AP^{-1}S^{[k]}_r P,\quad r=1,2.
\end{equation}
Here $P$ is the $n\times n$ matrix defined as in the stationary case:
\begin{enumerate}
\item The first $m$ columns of $P$ constitutes the vector $\bar{p}^0$ of given $n$ control points $p^0$  which are points in $\mathbb{R}^m$.
\item The last column is a column of $1$'s.
\item The rest of the columns are defined so that $P$ is non-singular. We assume here that the control points $p^0$ do not all lie on an $m-1$ hyper plane so that the first $m$ columns of $P$ will be linearly independent and that the column of $1$'s is independent of the first $m$ columns.
\end{enumerate}

\subsubsection{A Simpler SFS Construction for Non-Stationary Subdivision}

Before dealing with masks of increasing size, we present here a simpler SFS replacing the above SFS that was presented in \cite{LDV}. This simplification uses the theory in Section \ref{SMFS} which was not available in \cite{LDV}. The functions in this new SFS operate on row vectors in $\mathbb{R}^n$ by right matrix multiplication.

We use the same sequence of $n\times n$ matrices $\{S^{[k]}_r\}$ defined above and for a row vector  $A\in \mathbb{R}^n$ we define:
\begin{gather}
f_{r,1}\ :\ \mathbb{R}^n \to \ \mathbb{R}^m,\nonumber\\
f_{r,1}(A)=AS^{[1]}_r \bar{p}^0,\quad r=1,2,\label{eq5.5}
\end{gather}
and for $k>1$
\begin{gather}
f_{r,k}\ :\ \mathbb{R}^n \to \ \mathbb{R}^n,\nonumber\\
f_{r,k}(A)=AS^{[k]}_r ,\quad r=1,2.\label{eq5.6}
\end{gather}
The rows of $\bar{p}^0\in \mathbb{R}^{n\times m}$ are the $n$ initial control points in $\mathbb{R}^m$ for the subdivision process.

Let us follow a backward trajectory of the SFS $\Sigma := \{\mathcal{F}_k\}$, $\mathcal{F}_k=\{f_{1,k},f_{2,k}\}$, starting from $A\in \mathbb{R}^n$:
$$\mathcal{F}_k(A)=f_{1,k}(A)\cup f_{2,k}(A)=AS^{[k]}_1\cup AS^{[k]}_2,\ \ \ k>1,$$
and for $j,k>1$
$$\mathcal{F}_{j}(\mathcal{F}_k(A))=f_{1,j}(AS^{[k]}_1\cup AS^{[k]}_2)\cup f_{2,j}(AS^{[k]}_1\cup AS^{[k]}_2).$$
Noting that
\begin{equation}\label{ppm1}
f_{r,j}(AS^{[k]}_i)=AS^{[k]}_iS^{[j]}_r,
\end{equation}
it follows that the set generated at the $k$th step of a backward trajectory of $\Sigma$ is
\begin{equation}\label{ASSP}
\mathcal{F}_{1}\circ\mathcal{F}_2\circ ...\circ\mathcal{F}_{k-1}\circ\mathcal{F}_k(A)=\bigcup_{i_1,i_2,...,i_k\in \{1,2\}}AS^{[k]}_{i_k}...S^{[2]}_{i_2}S^{[1]}_{i_1}\bar{p}^0.
\end{equation}

For a $C^0$-stationary subdivision scheme, it is shown in \cite{SLG} that the maps defined by (\ref{f1f2k}) are contractive on $Q^{n-1}$, the $(n-1)$-dimensional hyperplane (flat) of vectors of the form $(x_1, \ldots, x_{n-1}, 1)$. The new maps \eqref{eq5.5} and \eqref{eq5.6} do not have an evident contraction property and yet a fixed-point theorem holds for  the backward trajectories of $\Sigma$. Instead of the flat $Q^{n-1}$ we consider another flat in $\mathbb{R}^n$, namely,
\begin{equation}\label{flatK}
K^{n-1}=\left\{(x_1,x_2,...,x_n) : \sum_{i=1}^n x_i=1\ \right\}.
\end{equation}
Notice that if the non-stationary scheme satisfies the constant reproduction property at every subdivision level, then all the maps in the SFS map $K^{n-1}$ into itself. That is,
\begin{equation}\label{SK2K}
\forall A\in K^{n-1}, \ \ \ AS_r^{[k]}\in K^{n-1}, \quad r=1,2,\ \ \ k\in\N.
\end{equation}
\begin{defn}[\bf The Set $K^{n-1}_C$]\label{Knm1C}
\[
K^{n-1}_C := \left\{(x_1,x_2,...,x_n) : \sum_{i=1}^n x_i=1,\ \ \ \  |x_i|\le C \ \right\}.
\]
\end{defn}

The following theorem establishes the convergence in the Hausdorff distance $h$ of the backward trajectories of $\Sigma$ to the limit curve of the non-stationary subdivision.

\begin{thm}\label{nsprop2}

Let $\{S_{a^{[k]}}\}$ be a non-stationary $C^0$-convergent subdivision scheme
and let $\Sigma=\{\mathcal{F}_k\}_{k=1}^\infty$ be the SFS defined in \eqref{eq5.5}  and \eqref{eq5.6}.
Then the backward trajectories of $\Sigma$ starting with ${\bf A}\subset K_C^{n-1}$ converge to a unique attractor which constitutes the limit curve (in $\mathbb{R}^m$) of the non-stationary scheme.
\end{thm}
\begin{proof}
Since $\{S_{a^{[k]}}\}$ converges, it immediately follows from (\ref{ASSP}), in view of (\ref{nspk}), that the backward trajectory of $\Sigma$ initialized with any $A$ converges. We would like to show that all the backward trajectories of $\Sigma$ initialized with an arbitrary set of points ${\bf A}\subset K_C^{n-1}$ converge to the same limit.
Since the subdivision scheme converges to a continuous function, it follows that
an infinite sequence $\eta=\{i_k\}_{k\in\N}$, $i_k\in\{1,2\}$, defines a vector  $\bar{q}_\eta$ of $n$ identical points in $\mathbb{R}^m$,
\begin{equation}\label{qeta}
\bar{q}_\eta=\lim_{k\to \infty}S^{[k]}_{i_k},...,S^{[2]}_{i_2}S^{[1]}_{i_1}\bar{p}^0=
\left( {\begin{array}{cc}
q_\eta \\
. \\
. \\
. \\
q_\eta \\
\end{array} } \right),\quad q_\eta=(q_{\eta,1},...,q_{\eta,m}),
\end{equation}
attached to a parametric value $x_\eta=\sum\limits_{k=1}^\infty (i_k-1)2^{-k}$. (See Remark 6.2 in \cite{LDV}.)
Starting the backward trajectory with any point $A\in {\bf A}$, and following the same sequence $\eta$, it follows from (\ref{ASSP}) that the limit is $A\bar{q}$. Recalling (\ref{flatK}), it now follows that $A\bar{q}_\eta=q_\eta$, $\forall A\in K^{n-1}$.

Hence,
\begin{equation}\label{PSSPa}
\lim_{k\to\infty}\mathcal{F}_{1}\circ\mathcal{F}_2\circ ...\circ\mathcal{F}_{k-1}\circ\mathcal{F}_k(A)=\bigcup_\eta A\lim_{k\to\infty}S^{[k]}_{i_k}...S^{[2]}_{i_2}S^{[1]}_{i_1}\bar{p}^0=\bigcup_\eta q_\eta,
\end{equation}
which is the set of all the limit points generated by the subdivision process starting with initial data $p^0$.
For a bounded set ${\bf A}$ this implies that
\begin{equation}\label{PSSPA}
\lim_{k\to\infty}\mathcal{F}_{1}\circ\mathcal{F}_2\circ ...\circ\mathcal{F}_{k-1}\circ\mathcal{F}_k({\bf A})=
\bigcup_{A\in {\bf A}}
\bigcup_\eta A\lim_{k\to\infty}S^{[k]}_{i_k}...S^{[2]}_{i_2}S^{[1]}_{i_1}\bar{p}^0=\bigcup_\eta q_\eta,
\end{equation}
\end{proof}

Above, we assumed that the number of control points $n$ is larger than the size of the subdivision masks. Therefore, when dealing with
non-stationary subdivision with increasing mask size, the definition of the related SFS should be revised.

\subsection{Non-Stationary Subdivision with Masks of Increasing Size}

Consider a binary non-stationary subdivision scheme with increasing mask size. Such schemes are suggested in the subdivision literature for generating highly smooth limit functions of small support \cite{DL}. Let us denote the support size of the $k$-th level subdivision mask  by $\sigma_k$. For example, the up-function which is a $C^\infty$-function of compact support is generated by a non-stationary subdivision with $\sigma_k=k+1$.
We hereby suggest a ``staircase'' SFS $\{\mathcal{F}_k\}$ where $\mathcal{F}_k=\big\{X_k; f_{1,k}, f_{2,k} \big \}$ with level dependent backward maps defined on row vectors $A\in \mathbb{R}^{n_k}$ at level $k$,
\begin{equation}\label{eq5.14}
f_{r,1}\ :\ \mathbb{R}^{n_0} \to \ \mathbb{R}^m,\ \ 
f_{r,1}(A)=AS^{[1]}_r \bar{p}^0,
\end{equation}
and for $k>1$
\begin{equation}\label{eq5.15}
f_{r,k}:\mathbb{R}^{n_{k}}\to \mathbb{R}^{n_{k-1}},\ \ 
f_{r,k}(A)=AS^{[k]}_r ,\quad r=1,2.
\end{equation}
The vector $p^0\in \mathbb{R}^{n_0\times m}$ represents the set of $n_0$ initial control points in $\mathbb{R}^m$ for the subdivision process.

The matrices $S^{[k]}_r$, $r=1,2$, $k\in\mathbb{N}$, are $n_{k}\times n_{k-1}$ matrices representing the subdivision rules at level $k$.
From a vector of $n_{k-1}$ values, the subdivision at level $k$ generates new $m_k>n_{k-1}$ values at level $k+1$. We set $n_k=\lceil\frac{m_k}{2}\rceil$
$S^{[k]}_1$ generates the first $n_{k}$ values and $S^{[k]}_2$ generates the last $n_{k}$ values.
We observe that by (\ref{eq5.15}) for $k>1$
\begin{equation}\label{FkH}
\mathcal{F}_k: \bH(\mathbb{R}^{n_{k}})\to \bH(\mathbb{R}^{n_{k-1}}).
\end{equation}

\begin{rem}[{Increase of Mask Size}]\label{Growth}
The growth rate  of the support sizes $\{\sigma_k\}$ should be limited so that the subdivision process starting with the initial control points $p^0$ defines a non-void curve in $\mathbb{R}^m$.
\end{rem}

The above construction leads to a sequence of backward maps $\{\mathcal{F}_k\}$ defined on a sequence of metric spaces
$\{ \bH(\mathbb{R}^{n_{k}})\}$ and to the question of convergence of the related backward staircase trajectories.
Let us develop the expression of the backward staircase trajectory.

For $A\subset \mathbb{R}^{n_{k}}$, let
\[
\mathcal{F}_k(A)=f_{1,k}(A)\cup f_{2,k}(A)=AS^{[k]}_1\cup AS^{[k]}_2.
\]

Let $\bar{A}=\{A_i\}$, $A_i\subset  K^{n_{i}-1}$, $i\in \mathbb{N}$,  be any base sequence for a backward staircase trajectory of $\Sigma=\{\mathcal{F}_k\}$. It follows that
\begin{equation}\label{ASSP2}
p_k(\bar{A})=\mathcal{F}_{1}\circ\mathcal{F}_2\circ ...\circ\mathcal{F}_{k-1}\circ\mathcal{F}_k(A_{k})=\bigcup_{i_1,i_2,...,i_k\in \{1,2\}}A_{k}S^{[k]}_{i_k}...S^{[2]}_{i_2}S^{[1]}_{i_1}\bar{p}^0.
\end{equation}

Hence, if the non-stationary subdivision converges, then many staircase trajectory converges.
For example, the backward trajectory with the base sequence elements
\begin{equation}\label{base1}
A_i=(1,0,0,...,0),\ \ \ i\in \mathbb{N}.
\end{equation}
Using Corollary \ref{STSFS} in order to show convergence of all trajectories to the same limit, one should check the related Lipschitz constants $\{s_i\}$. For the test case of a staircase subdivision generating the up-function, the conditions on $\{s_i\}$ stated in Corollary \ref{STSFS} are not satisfied.
It can be shown that, for $\ell$ large enough, a map grouping of order $\ell$ gives contractive maps and hence the backward trajectories converge.
We present below another approach for showing convergence of all staircase backward trajectories to the same limit. This approach extends the idea developed in Theorem \ref{nsprop2} and it can be used for all schemes whose mask sizes $\{\sigma_k\}$ exhibit polynomial growth.

\begin{thm}\label{nspropkv}
Let $\{S_{a^{[k]}}\}$ be a non-stationary $C^1$-convergent subdivision scheme
and let $\Sigma=\{\mathcal{F}_k\}_{k=1}^\infty$ be the SFS defined in \eqref{eq5.14} and \eqref{eq5.15}.
Then the backward trajectories of $\Sigma$ starting with any base sequence $\bar{A}=\{A_k\}$  with $A_k\subset K^{n_k-1}_C$, for some constant $C$, converge to a unique attractor which constitutes the limit curve (in $\mathbb{R}^m$) of the non-stationary scheme.
\end{thm}

\begin{proof}
The idea of the proof is similar to that of the proof of Theorem \ref{nsprop2}. However,  we do not have a relation like (\ref{qeta}) since the dimension of $S^{[k]}_{i_k}...S^{[2]}_{i_2}S^{[1]}_{i_1}\bar{p}^0$ increases with $k$. Let us assume that $\{\sigma_k\}$ is growing algebraically with $k$. 
Since the subdivision scheme converges to a continuous function, it follows that
an infinite sequence $\eta=\{i_k\}_{k\in\N}$, $i_k\in\{1,2\}$, defines a limit point 
$q_\eta\in \mathbb{R}^m$ attached to a parametric value $x_\eta=\sum\limits_{k=1}^\infty (i_k-1)2^{-k}$. Let us denote by $\bar{q}_\eta^{n_k}$ the vector of $n_k$ identical points in $\mathbb{R}^m$,
\begin{equation}\label{qetank}
\bar{q}_\eta^{n_k}=
\left( {\begin{array}{cc}
q_\eta \\
. \\
. \\
. \\
q_\eta \\
\end{array} } \right),\quad q_\eta=(q_{\eta,1},...,q_{\eta,m}).
\end{equation}
For $k$ large enough and with $\eta_k=\{i_j\}_{j=1}^k$, $i_j\in\{1,2\}$, the vector of $n_k$ points in $\mathbb{R}^m$
\[
\bar{p}_{\eta_k}=S^{[k]}_{i_k}...S^{[2]}_{i_2}S^{[1]}_{i_1}\bar{p}^0\ ,
\]
is close to the smooth limit curve generated by the subdivision scheme.
To estimate how close we use here some observations from the theory of subdivision \cite{DL}.

For a $C^1$-convergent scheme it can be shown that
$$\|q_{\eta_k}-q_{\eta}\| = \cO(2^{-k}),$$
$$\|(\bar{p}_{\eta_k})_{i,\cdot}-q_{\eta_k}\| = n_k\cO(2^{-k}),\ \ i=1,...,n_k.$$
Hence, we obtain
\begin{equation}\label{O2k}
\bar{p}_{\eta_k}=\bar{q}_{\eta}^{n_k}+n_k\cO(2^{-k}),\ \ \quad\text{as}\  k\to \infty.
\end{equation}
Since $A_k\in K^{n_k-1}_C$, it follows that $A_k\bar{q}_{\eta}^{n_k}=q_\eta$ and an additional $n_k$ factor enters when we apply $A_k$ to the remainder term in (\ref{O2k}):
\begin{equation}\label{O2ka} 
A_k\bar{p}_{\eta_k}=q_\eta+(n_k)^2 \cO(2^{-k}) ,\quad\text{as $k\to \infty$}.
\end{equation}
Now, $\{\sigma_k\}$ is growing algebraically with $k$ and  $n_k=n_0+\sigma_k-\sigma_0$ and thus for any base sequence $\bar{A}=\{A_k\}$ with $A_k\subset K^{n_k-1}_C$ we get
\begin{equation}\label{Aktoeta}
A_{k}S^{[k]}_{i_k}...S^{[2]}_{i_2}S^{[1]}_{i_1}\bar{p}^0\to q_\eta,\quad\text{as $k\to \infty$}.
\end{equation}
The proof now follows by using (\ref{ASSP2}).
\end{proof}

\section{Non-Uniform Subdivision}\label{NUSTA}

Another class of subdivision processes includes non-uniform schemes in which the subdivision rules depend on the location. An example of such schemes is dealt with in \cite{DLY}.
The approach used in \cite{LDV} cannot be used for non-uniform schemes not even those with a fixed mask size. The introduction of the notion of staircase trajectories opens up new possibilities. Another option is to use the structure of trees of maps.  

Let us recall that the challenge is to show that the limit of the subdivision process can be represented as the unique fixed point of some iterative process starting from a certain set of starting points (or base sequences). Below, we present two ways of approaching the non-uniform case.

\subsection{First Approach - Using Staircases of Maps}

Consider a general {\bf linear} subdivision process, univariate or multivariate, stationary or non-stationary, uniform or non-uniform. Here we restrict he discussion to univarate binary subdivision. Starting with $n_0$ control points $p^0\in\mathbb{R}^m$, there is  an $n_1\times n_0$ matrix $S^{[1]}$ that generates all the $n_1$ points derived at level 1 of the subdivision process
\[
\bar{p}^1=S^{[1]}\bar{p}^0.
\]
Inductively, denote by $S^{[k]}$ the $n_k\times n_{k-1}$ matrix which represents the subdivision rules transforming the $n_{k-1}$ points $p^{k-1}$ attained at level $k-1$ to the $n_k$ points $p^k$ (in $\mathbb{R}^m$) at level $k$:
\[
\bar{p}^k=S^{[k]}\bar{p}^{k-1}.
\]
Note that here $n_k=\cO(2^k)$.
We have a sequence of metric spaces $\{\R^{n_k}\}$ (endowed with say the discrete $L^2$-metric) and a sequence of maps between the metric spaces. If the subdivision process is known to be $h$-convergent,
then the forward trajectory starting with $p^0$ converges to $p^\infty\subset\mathbb{R}^m$.
However, different initial vectors yield different limit points. As done in the previous section, let us view the matrices $\{S^{[k]}\}$ as backward maps.
For $k=1$,
\begin{equation}
S^{[1]}\ :\ \mathbb{R}^{\ell\times n_1}\ \to \ \mathbb{R}^{\ell\times m},\ \ \
A\in \mathbb{R}^{\ell\times n_1} \to AS^{[1]}\bar{p}^0 \in \mathbb{R}^{\ell\times m},
\end{equation}
and for $k>1$, 
\begin{equation}
S^{[k]}\ :\ \mathbb{R}^{\ell\times n_k}\ \to \ \mathbb{R}^{\ell\times n_{k-1}},\ \ \
A\in \mathbb{R}^{\ell\times n_k} \to AS^{[k]} \in \mathbb{R}^{\ell\times n_{k-1}}.
\end{equation}
Of course, the above maps are only well defined for linear subdivision schemes.

\subsubsection{The Choice of a Base Sequence}

The question now is: For which class of base sequences $\bar{A}=\{A_k\}$
do all backward trajectories of the above maps converge to $p^\infty$? The choice
\[
A_k=I_{n_k\times n_k},
\]
the identity matrix in $\R^{n_k\times n_k}$, generates the sequence $\{\bar{p}^k\}$ as its trajectory. Clearly, the backward trajectory with respect to this base sequence converges to $p^\infty$. 

\begin{defn}\label{Krhom1}

Let $\rho$ be a positive integer and let $C$ be a positive constant. For $n_k>\rho$, we denote by $K^{n_k,\rho-1}_C$ the set of all row vectors of length $n_k$ of the form
\begin{equation}\label{v}
(0,0,...,0,v,0,0,...,0),\ \ \ v\in K^{\rho-1}_C.
\end{equation}
\end{defn}

\begin{defn}[\bf Elements of the Base Sequence]\label{Arhom1}

Let $\rho$ be a positive integer, $C$ a positive constant, and $n_k>\rho$. We denote by $A^{\ell\times n_k,\rho}_C$ the set of all $\ell\times n_k$ matrices with rows in  $K^{n_k,\rho-1}_C$ such that the sum of the rows does not have zero elements.
\end{defn}

 Let us further assume that the subdivision process converges to a continuous limit. Then, in the spirit of Theorem \ref{nsprop2}, we can prove the following.
 
\begin{thm}\label{nu1}
Consider a non-uniform $C^\nu$-convergent subdivision scheme, $\nu>0$,
and let $\{S^{[k]}\}_{k=1}^\infty$ be the backward maps defined above.
Then the backward trajectories starting with any base sequence $\bar{A}=\{A_k\}$  with $A_k\in A^{n_k\times n_k,\rho}_C$ for some fixed $C$ and $\rho$, converge to a unique attractor whose rows constitute the limit curve (in $\mathbb{R}^m$) of the non-uniform scheme.
\end{thm}
\begin{proof}
Unlike the case in Theorem \ref{nsprop2}, here the vectors of points $\bar{p}^k$ are of fast increasing length, and as such, even if the subdivision is convergent, the vectors are not converging to a vector with identical rows. However, let us consider any partial vector $w$ of $\bar{p}^k$ containing $\rho$ consecutive rows of
$\bar{p}^k$. Such vectors tend to constant vectors as $k\to \infty$. Left multiplication of 
$\bar{p}^k$ by a row vector of the form (\ref{v}) corresponds to left multiplication of   $w$ by $v$. The result, as in Theorem \ref{nsprop2}, would be a point in $\mathbb{R}^m$ approaching
$p^\infty$. The rows of the matrices in the base sequences are all of the form (\ref{v}).
The condition that  the sum of the rows does not have zero elements ensures that 
$A_k\bar{p}^k$ generates points in $\mathbb{R}^m$ which tend to a dense set of points in $p^\infty$.
\end{proof}

\subsection{Second Approach - Using Trees of Maps}

To present the idea, let us consider a univariate binary non-uniform subdivision scheme of fixed mask size.
Being non-uniform means that the subdivision masks depend upon the geometric location of the control points, or, on the parametric correspondence of the control points at all levels.
In other words, in order to generate the $n_k$ points $\{p^k\}$ at level $k$, we use $n_k$ different masks.
As in the uniform case, we would like to attach to the subdivision process a related SFS, a binary SFS.
We already know how to define a function system to a given linear subdivision rule. 
The question is, how to introduce a parametric correspondence into the definition of the system function systems and to the SFS iterations?

As in Section \ref{TM}, we denote by $B^{[k]}$ the set of binary sequences of length $k$, i.e.,
\begin{equation}\label{Ek}
B^{[k]}=\{\eta_k : \eta_k=\{i_j\}_{j=1}^k,\ i_j\in\{1,2\}\}.
\end{equation}
A parametric location at level $k$ is determined by a binary sequence $\eta_k\in B^{[k]}$. In the non-uniform case the matrices $S_1^{[k]}$ and $S_2^{[k]}$ are location dependent.
Instead of two matrices at each level we have $2^k$ matrices at level $k$, recursively defined as follows.
\begin{defn}[\bf Recursive Definition of the Matrices $S^{[k]}_{\eta_k}$]\label{Sketak}

$S_1^{[1]}$ and $S_2^{[1]}$ are defined as in the uniform case, i.e., given the $n_{0}$ points, $p^0$, at level zero, $S^{[1]}_1$ generates the first $n_{0}$ points at level 1 and $S^{[1]}_2$ generates the last $n_{0}$ points. We denote by $p_{\eta_k}$ the $n_0$ points at level $k$ attached to a parametric location $\eta_{k}$:
\begin{equation}\label{petak2}
\bar{p}_{\eta_k}=S^{[k]}_{i_1i_2...i_k}\cdot\cdot\cdot S^{[3]}_{i_1i_2i_3}S^{[2]}_{i_1i_2}S^{[1]}_{i_1}\bar{p}^0\ .
\end{equation}
For $k>1$, given the vector of points $\bar{p}_{\eta_k}$, we define two $n_0\times n_0$ matrices $S^{[k+1]}_{i_1i_2...i_k1}$ and $S^{[k+1]}_{i_1i_2...i_k2}$. The first one generates the first $n_0$ points resulting from $\bar{p}_{\eta_k}$ and the second generates the last $n_0$ points resulting from $\bar{p}_{\eta_k}$.
\end{defn}

We hereby define a tree of maps $\{f_{\eta_k}\}$: For $k=1$, let 
\begin{equation}\label{f1f2-1v} 
f_{r}\ :\ \mathbb{R}^{n_0} \to \ \mathbb{R}^m,\ \ \
f_{r}(A) := AS^{[1]}_r \bar{p}^0,\quad r=1,2,
\end{equation}
where $\bar{p}^0\in \mathbb{R}^{n_0\times m}$ represents the set of $n_0$ initial control points in $\mathbb{R}^m$ for the subdivision process. 

For $k>1$, let
\begin{equation}\label{f1f2-kv}
f_{\eta_k}\ :\ \mathbb{R}^{n_0} \to \ \mathbb{R}^{n_0},\ \ \
f_{\eta_k}(A) := AS^{[k]}_{\eta_k}.
\end{equation}

Consider the space $X :=\mathbb{R}^{n_0}$ and the infinite binary tree of maps on $X$ defined as above in (\ref{f1f2-1v}) and (\ref{f1f2-kv}).
Note that for $k>1$, $f_{\eta_k} : X\to X$, while for $k=1$, $f_{\eta_k} : X\to \mathbb{R}^m$.
Thus, convergent backward trajectories along paths in the tree converge to points in $\mathbb{R}^m$.
Using the same steps as in the proof of Theorem \ref{nsprop2} for each path of the tree, the following theorem holds. 

\begin{thm}\label{nu2}
Consider a linear binary non-uniform $C^0$-convergent subdivision scheme
and let $TM$ be the tree of maps defined above.
Then the backward trajectories along all paths in the tree converge for any $A\in K^{n_0-1}$, and the attractor $U_{TM}$ of the tree constitutes the limit curve (in $\mathbb{R}^m$) of the non-uniform scheme.
\end{thm}

\section{Summary and Conclusions}
We considered countable sequences of maps on a complete metric space $(X,d)$ and introduced the novel concept of an infinite tree of maps $X\to X$. Conditions for the convergence of these trees of maps to a unique attractor were derived. More generally, we investigated countable sequences of complete metric spaces $\{(X_i,d_i)\}$ and an associated countable sequence of maps $T_i :X_i\to X_{i-1}$. We provided conditions for the convergence of the backward trajectories $\Psi_k = T_1\circ \cdots \circ T_k$, $k\in \N$, to a unique attractor. As an example, we considered trees of maps arising from iterated function systems and constructed a new class of fractals that are both scale dependent and location dependent. Finally, we exhibited the connections to non-stationary and non-uniform subdivision schemes and showed that this new approach is capable of linking all types of linear subdivision schemes to attractors of sequences of function systems.

\end{document}